\documentclass[12pt]{article}
\usepackage{amssymb}
\usepackage{amsmath,bm}
\usepackage{graphics,mathrsfs}
\usepackage{graphicx,epsfig}
\usepackage{subfigure}
\usepackage{setspace}
\usepackage{enumerate}
\usepackage{enumitem}
\usepackage{varioref}
\usepackage{natbib}
\usepackage{color}
\setcitestyle{authoryear,open={(},close={)}}
\usepackage{amsthm}
\makeatletter \oddsidemargin  -.1in \evensidemargin -.1in
\begin{document}
\newcommand{\bea}{\begin{eqnarray}}
\newcommand{\eea}{\end{eqnarray}}
\newcommand{\nn}{\nonumber}
\newcommand{\bee}{\begin{eqnarray*}}
\newcommand{\eee}{\end{eqnarray*}}
\newcommand{\lb}{\label}
\newcommand{\nii}{\noindent}
\newcommand{\ii}{\indent}
\newtheorem{theorem}{Theorem}[section]
\newtheorem{example}{Example}[section]
\newtheorem{counterexample}{Counterexample}[section]
\newtheorem{corollary}{Corollary}[section]
\newtheorem{definition}{Definition}[section]
\newtheorem{lemma}{Lemma}[section]
\newtheorem{remark}{Remark}[section]
\newtheorem{proposition}{Proposition}[section]
\renewcommand{\theequation}{\thesection.\arabic{equation}}
\renewcommand{\labelenumi}{(\roman{enumi})}
\title{\bf Failure extropy, dynamic failure extropy and their weighted versions}
\author{Suchandan Kayal\thanks {Email address:
		kayals@nitrkl.ac.in,~suchandan.kayal@gmail.com}
                    \\{\it \small Department of Mathematics, National Institute of
                    	Technology Rourkela, Rourkela-769008, India              }\\\\
                    {\it \small Revised version to appear in ``Stochastics and Quality Control"  (@ DE GRUYTER)           }
                    }
\date{}
\maketitle
\begin{center}
{\large \bf Abstract}
\end{center}
Extropy was introduced as a dual complement of the Shannon entropy (see \cite{lad2015extropy}). In this investigation, we consider failure extropy and its dynamic version. Various basic properties of these measures are presented. It is shown that the dynamic failure extropy characterizes the distribution function uniquely. We also consider weighted versions of these measures. Several virtues of the weighted measures are explored. Finally, nonparametric estimators are introduced based on the empirical distribution function.
\section{Introduction}
Since its introduction, the Shannon entropy has been of primary interest in various applied fields of information theory. Recently, \cite{lad2015extropy} showed that a complementary dual function of the Shannon entropy exists. They called it extropy. Let $X$ be a nonnegative and absolutely continuous random variable with probability density function (pdf) $f$, cumulative distribution function (cdf) $F$ and survival function (sf) $\bar{F}=1-F.$ The extropy of $X$ is defined as (see \cite{lad2015extropy})
\begin{eqnarray}\label{eq1.1}
\mathcal{E}(X)=-\frac{1}{2}\int_{0}^{\infty}f^{2}(x)dx=-\frac{1}{2}E\left[f(X)\right].
\end{eqnarray}
From (\ref{eq1.1}), it is clear that $\mathcal{E}(X)\in[-\infty,0).$ The concept of extropy is usually applied to score the forecasting distributions based on the total log scoring method. We refer to \cite{gneiting2007strictly}, \cite{furuichi2012mathematical} and \cite{vontobel2012bethe} for further applications of extropy. \cite{qiu2017extropy} studied various properties of extropy for order statistics and record values.
Similar to the residual entropy, \cite{qiu2018residual} proposed extropy for the residual random variable. They studied some properties including monotonicity and characterization of the residual extropy. \cite{kamari2020extropy} introduced past extropy and obtained characterization result for the past extropy of the largest order statistic.

The Shannon differential entropy has some drawbacks. One of these is that it is defined for
the distributions having density functions. For example, the definition of the Shannon differential
entropy is not applicable for a mixture density comprised of a combination of Guassians and delta functions.
Due to this, \cite{rao2004cumulative} introduced cumulative residual entropy by substituting  sf in place of the pdf in the definition of the differential entropy. \cite{di2009cumulative} proposed cumulative entropy based on the cdf and investigated various properties. Motivated by the way, the cumulative residual entropy was developed, \cite{zografos2005survival} proposed the notion of the generalized survival entropy, which parallels the Renyi entropy. Inspired by the generalized survival entropy, \cite{abbasnejad2011some} introduced generalized failure entropy and discussed various properties with some characterization results. Very recently, \cite{abdul2019dynamic} borrowed the notion of the survival entropy in order to propose survival extropy of a nonnegative and absolutely continuous random variable. It is given by
\begin{eqnarray}\label{eq1.2}
\mathcal{E}_{s}(X)=-\frac{1}{2}\int_{0}^{\infty}\bar{F}^{2}(x)dx.
\end{eqnarray}
This measure was constructed after replacing pdf by sf  in  (\ref{eq1.1}). In this paper, we introduce failure extropy, dynamic failure extropy and their weighted versions. We point out that the failure extropy has been proposed similarly to the cumulative entropy (see \cite{di2009cumulative}) by substituting cdf in place of the pdf in (\ref{eq1.1}). Before presenting main results, we recall some notions of various stochastic orders. These are useful in deriving some ordering results based on the proposed measures in the subsequent sections. Let $X$ and $Y$ be two nonnegative and absolutely continuous random variables with pdfs $f$ and $g$, cdfs $F$ and $G$, sfs $\bar{F}=1-F$ and $\bar{G}=1-G,$ respectively. The hazard and reversed hazard rates of $X$ and $Y$ are respectively denoted by $h_{X},h_{Y}$ and $r_{X},r_{Y}.$ Further, denote the right continuous inverses of $F$ and $G$ by $F^{-1}$ and $G^{-1}$, respectively.
\begin{definition}
The random variable $X$ is said to be smaller than $Y$ in the
\begin{itemize}
	\item usual stochastic order, denoted by
	$X\le_{st}Y$, if $F(x)\ge G(x)$, for all $x$ in
	the set of real numbers;
	\item dispersive order, denoted by $X\le_{disp}Y$, if
	$F^{-1}(\alpha)-F^{-1}(\beta)\le
	G^{-1}(\alpha)-G^{-1}(\beta)$, for all $0<\beta\le
	\alpha<1;$
	\item hazard rate order, denoted
	by $X\le_{hr}Y$, if $h_{X}(x)\ge h_{Y}(x)$, for all
	$x>0$;
	\item reversed hazard rate order, denoted by $X\le_{rh}Y$, if  $r_{X}(x)\le r_{Y}(x)$, for all $x>0$;
	\item likelihood ratio order, denoted by $X\le_{lr}Y$, if $g(x)/f(x)$ is increasing with respect to $x>0.$
\end{itemize}
\end{definition}
The terms increasing and decreasing are used in non-strict sense.
The paper is organized as follows. In Section $2$, the failure extropy is proposed. We study various properties of it. In addition, the concept of the bivariate failure extropy is introduced. Section $3$ studies dynamic form of the failure extropy. Several properties similar to the failure extropy are investigated. It is shown that the dynamic failure extropy uniquely determines the distribution function. Further, a characterization for the power distribution is provided in Section $3$. Weighted versions of the failure extropy and the dynamic failure extropy are considered in Section $4.$  Properties related to the weighted measures are discussed. Nonparametric estimators of the measures are proposed in Section $5$. Real data sets are considered for the purpose of computation of the proposed estimators. Finally, concluding remarks are provided in Section $6.$

\section{Failure extropy\setcounter{equation}{0}}
In this section, we study some properties of the
failure extropy. The definition of extropy, proposed here, is analogous to the definition of the cumulative entropy by \cite{di2009cumulative}. Consider a nonnegative
absolutely continuous random variable $X$ with support $S_{X},$
cdf $F$ and
pdf $f$. The
failure extropy of $X$, denoted by $\mathcal{E}_{f}(X)$ is defined as (see also \cite{kundu2020cumulative} and \cite{nair2020dynamic})
\begin{eqnarray}\label{eq2.1}
\mathcal E_f(X)=-\frac{1}{2} \int_0^{\sup S_{X}} F^2(x)\mathrm dx.
\end{eqnarray}
Note that the above equation has been obtained after substituting the cdf of $X$ in place of the pdf in (\ref{eq1.1}). Like extropy, the failure extropy given by (\ref{eq2.1}) is always negative. Below, we obtain closed-form expressions for the
failure extropy of various distributions.
\begin{example}~~\vspace{-.5cm}\\
	\begin{itemize}
\item Let $X$ have uniform distribution in the interval
$(0,b)$. Then, $\mathcal{E}_{f}(X)=-b/6.$
\item Suppose $X$ has Type III extreme value
distribution with cdf
$F(x)=\exp\{\alpha(x-\beta)\},~x<\beta,~\alpha>0.$
Then,
$\mathcal{E}_{f}(X)=-(e^{2\beta\alpha}-1)/(4\alpha
e^{2\beta\alpha}).$
\item Let us assume that $X$ follows power distribution
with cdf $F(x)=x^{a},~0<x<1,~a>0.$ Then,
$\mathcal{E}_{f}(X)=-1/(2(2a+1))$.
   \end{itemize}
\end{example}
\begin{remark}
If the support of the absolutely continuous random variable is unbounded, then the failure extropy is equal to $-\infty.$
\end{remark}

In the following, we provide  lower as well as upper bounds of the
failure extropy when a random variable has bounded support. The lower bound can be achieved using
$F^{2}(x)\le F(x).$ On the other hand, the upper bound can be obtained utilizing Jensen's inequality for integrals.
\begin{proposition}\label{prop2.1}
For a nonnegative and absolutely continuous
random variable $X$ with support $[0,u]$ and cdf $F$, we have $(u<\infty)$
\begin{itemize}
\item $\mathcal{E}_{f}(X)\ge-\frac{1}{2}\int_{0}^{u}F(x)dx=-\frac{1}{2}[u-E(X)].$
\item $\mathcal{E}_{f}(X)\le-\frac{1}{2u}\left(\int_{0}^{u}F(x)dx\right)^{2}=-\frac{1}{2u}[u-E(X)]^{2}.$
\end{itemize}
\end{proposition}
\begin{remark}
	Let $X$ follow uniform distribution in the interval (0,u). Then, from the above proposition,  lower and upper bounds of the extropy of uniform distribution can be obtained as $-u/4$ and $-u/8$, respectively. Clearly, $-u/4<
\mathcal{E}_{f}(X)(=-u/6)<-u/8.$
\end{remark}

Next, we take monotone transformations and examine
its effect on the failure extropy.
\begin{theorem}\label{th2.1}
Consider a nonnegative absolutely continuous
random variable $X$ with cdf $F$. Take $\psi:S\rightarrow T$, a strictly
monotone and differentiable function, where $S$ is the support of $X$ and $T$ is an interval, say $[u,v]$.   Denote the failure extropy
of $Y=\psi(X)$ by $\mathcal{E}_{f}(Y).$ Then,
\begin{eqnarray*}
\mathcal{E}_{f}(Y)=\begin{cases}
-\frac{1}{2}\int_{\psi^{-1}(u)}^{\psi^{-1}(v)}\psi'(x)F^{2}(x)dx, &
if~\psi~
\mbox{is strictly increasing;} \\
-\frac{1}{2}\int_{\psi^{-1}(v)}^{\psi^{-1}(u)}|\psi'(x)|F^{2}(x)dx,
& if~\psi~ \mbox{is strictly decreasing.}
\end{cases}
\end{eqnarray*}
\end{theorem}
\begin{proof}
We assume that $\psi$ is strictly increasing. The
cdf of $Y$ is $G(y)=F(\psi^{-1}(y)).$ As a
result, from the definition of the failure
extropy, we have
\begin{eqnarray}\label{eq2.2*}
\mathcal{E}_{f}(Y)=-\frac{1}{2}\int_{u}^{v}G^{2}(y)dy=-\frac{1}{2}\int_{u}^{v}F^{2}
(\psi^{-1}(y))dy.
\end{eqnarray}
Now, utilizing the inverse transform
$x=\psi^{-1}(y)$ in (\ref{eq2.2*}), the first part
follows. The second part can be proved similarly and is omitted. This completes the proof.
\end{proof}
The following corollary is immediate from
Theorem \ref{th2.1}. It shows that the failure
extropy is a shift-independent measure.
\begin{corollary}\label{cor2.1}
If $Y=aX+b$, where $a>0$ and $b\ge0,$ then $\mathcal{E}_{f}(Y)=a\mathcal{E}_{f}(X).$
\end{corollary}

Let us now consider three nonnegative random
variables $X,~X_{\tau}^*$ and $\tau X$, where
$\tau>0.$ The cdf of $X_{\tau}^*$ is taken as
$F^{*}(x)=[F(x)]^{\tau}.$ This is known as the
proportional reversed hazard rate model. Now, we obtain inequality among the
failure extropies of $X,~X_{\tau}^*$ and $\tau
X$ (see also \cite{nair2020dynamic}).
\begin{proposition}\label{prop2.2}
The inequalities $\mathcal{E}_{f}(X^*_\tau)\le
\mathcal{E}_{f}(X)\le \mathcal{E}_{f}(\tau X)$
hold when $\tau\in(0,1]$. The inequalities get
reversed when $\tau\in[1,\infty).$
\end{proposition}
\begin{proof}
The proof follows from Corollary \ref{cor2.1} and
the result $F^{2}(x)\le (\ge) F^{2\tau}(x)$
when $0<\tau\le1~(\tau\ge1)$.
\end{proof}
\begin{remark}
Let $X_{1},\ldots,X_{n}$ be a random sample drawn
from a population with cdf $F.$ Denote the
largest order statistic of the sample by
$X_{n:n}$. Further, it is known that the cdf of
$X_{n:n}$ denoted by $F_{n:n}(x)$ is equal to
$[F(x)]^{n},$ where $n$ is a positive integer.
Thus, as an application of Proposition
\ref{prop2.2}, we have $\mathcal{E}_{f}(n
X_{1})\le \mathcal{E}_{f}(X_{n:n})$.
\end{remark}

The notions of partial orderings between the random variables have been introduced in the
literature. These have useful applications in the theory of reliability and economics (see \cite{shaked2007stochastic}). Here, we
propose a partial preorder, which helps to find out a better system in the sense of less uncertainty. Let
$X$ and $Y$ be two nonnegative absolutely continuous random variables with cdfs $F$ and $G$, respectively.
Next, let us define uncertainty order based on the failure extropy.
\begin{definition}
A random variable $X$ is said to be smaller than
$Y$ in the sense of the
failure extropy, if $\mathcal{E}_{f}(X)\le
\mathcal{E}_{f}(Y).$ We denote $X\le_{fe} Y$.
\end{definition}
It is easy to see that the order defined above is a partial preorder, since it is reflexive, antisymmetric and transitive. There are many distributions for which the failure
extropy order holds. For example, let us consider two random variables $X$ and $Y$ with respective cdfs $F(x)=x^{a_{1}}$
and $G(x)=x^{a_{2}}$, $0<x<1$. Further, let
$0<a_{1}<a_{2}.$ Then, it can be clearly shown
that $\mathcal{E}_{f}(X)\le \mathcal{E}_{f}(Y)$,
that is, $X\le_{fe} Y$. Now, the question
arises. Is there any relation between
the usual stochastic order and the failure
extropy order? The answer is yes. The following
implication can be easily obtained from the definitions of the usual stochastic and failure extropy orders (see \cite{nair2020dynamic})
\begin{eqnarray}\label{eq2.3*}
X\le_{st}Y\Rightarrow X\le_{fe}Y.
\end{eqnarray}
Let $\{X_{1},\ldots,X_{m}\}$ be a collection of independent and identically distributed random observations. Denote the $l$th order statistic by $X_{l:m}$. It is not difficult to see that $X_{l:m}\le_{st}X_{p:m}$, for all $1\le l<p\le m.$ For samples with possibly unequal sample sizes, we have $X_{m:m}\le_{lr}X_{m+1:m+1}$ and $X_{1:m+1}\le_{lr}X_{1:m}$ (see \cite{shaked2007stochastic}). Again, the likelihood ratio order implies the usual stochastic order. Thus, from (\ref{eq2.3*}), we have the following observations:
\begin{itemize}
	\item $\mathcal{E}_{f}(X_{l:m})\le \mathcal{E}_{f}(X_{p:m}),$ for all $1\le l<p\le m;$
	\item $\mathcal{E}_{f}(X_{m:m})\le \mathcal{E}_{f}(X_{m+1:m+1});$
	\item $\mathcal{E}_{f}(X_{1:m+1}) \le \mathcal{E}_{f}(X_{1:m})$.
\end{itemize}
Further, it is well known that
$X\le_{disp}Y\Rightarrow X\le_{st}Y.$ Thus, from (\ref{eq2.3*}), the extended implication chain is
\begin{eqnarray}\label{eq2.4}
X\le_{disp}Y\Rightarrow X\le_{st}Y\Rightarrow
X\le_{fe}Y.
\end{eqnarray}

It can be seen that the variance and differential entropy act additively when they are computed for the sum of two independent random variables. This property is not satisfied by the survival extropy. Below, we show that for two independent random variables, a lower bound of the failure extropy of the sum is the maximum of the individual failure extropies (also see \cite{kundu2020cumulative})
\begin{proposition}
Consider two independent random variables $X$ and
$Y$ having log-concave density functions. Then,
$$\mathcal{E}_{f}(X+Y)\ge
\max\{\mathcal{E}_{f}(X),\mathcal{E}_{f}(Y)\}.$$
\end{proposition}
\begin{proof}
Under the assumption that $X$ has log-concave
density function and using Theorem $3.B.7$ of
\cite{shaked2007stochastic}, we obtain
$X\le_{disp} X+Y$. Thus, from (\ref{eq2.4}),
$\mathcal{E}_{f}(X)\le \mathcal{E}_{f}(X+Y).$
Further, considering the case that $Y$ has
log-concave density function, we get
$\mathcal{E}_{f}(Y)\le \mathcal{E}_{f}(X+Y).$
Combining these inequalities, we obtain the required
result.
\end{proof}

Next, we derive conditions under which the failure extropy of two random variables can be ordered. It is known that a random variable $X$ has decreasing
failure rate (DFR) if its failure rate function
is decreasing. We have
$X\le_{disp}Y$ under the conditions that $X\le_{hr}Y$ and $X$ or
$Y$ is DFR. Thus, from (\ref{eq2.4}), we have the
following result.
\begin{proposition}
If $X\le_{hr}Y$ and $X$ or $Y$ is DFR, then
$\mathcal{E}_{f}(X)\le \mathcal{E}_{f}(Y).$
\end{proposition}

In the following proposition, we show that the failure extropy of $X$ can
be expressed in terms of the expectations of $X$
and $Y$. Denote
$R(x)=-\frac{1}{2}\int_{x}^{\infty}F(u)du.$
\begin{proposition}
Consider two nonnegative and absolutely
continuous random variables $X$ and $Y$ having
unequal means and let $X\le_{st}Y.$  If $
R(x)<\infty$ and $E[R(Y)]<\infty$, then
$$\mathcal{E}_{f}(X)=E[R(Y)]+E[R'(V)][E(Y)-E(X)],$$
where $V$ is nonnegative and absolutely
continuous random variable with pdf
$$k(v)=\frac{\bar{G}(v)-\bar{F}(v)}{E(Y)-E(X)},~v>0$$.
\end{proposition}
\begin{proof}
The failure extropy can be expressed as
$\mathcal{E}_{f}(X)=E(R(X)).$ Now, the rest of
the proof follows from \cite{di1999probabilistic}. Thus, it is omitted.
\end{proof}

In this part of the section, we introduce failure extropy of bivariate random vector. Consider a bivariate random vector $(X,Y)$ with joint
pdf $k$ and cdf $K$. The supports of $X$ and $Y$ are denoted by $S_{X}$ and $S_{Y}$, respectively. Further, assume
that the marginal pdfs and cdfs of $X$ and $Y$
are denoted by $f,g$ and $F,G$, respectively. The
bivariate failure extropy of $(X,Y)$ is defined
as
\begin{eqnarray}\label{eq2.2}
\mathcal{E}_{f}(X,Y)=\frac{1}{4}\int_{0}^{\sup S_{X}}\int_{0}^{\sup S_{Y}}K^{2}(x,y)dxdy.
\end{eqnarray}
Let $X$ and $Y$ be independent, that is,
$K(x,y)=F(x)G(y)$. Utilizing this in
(\ref{eq2.2}), we get
\begin{eqnarray}\label{eq2.3}
\mathcal{E}_{f}(X,Y)=\mathcal{E}_{f}(X)\mathcal{E}_{f}(Y).
\end{eqnarray}
This relation reveals that when $X$ and $Y$ are
independent, the bivariate failure extropy
has the additive property like the
Shannon entropy. Now, we show that the bivariate failure extropy is not invariant under nonsingular transformations. Note that transformations play an important role to transform
a distribution into another. So, it is always of interest to see the form of the bivariate failure extropy
of the transformed random variables. Here we have studied the form of the bivariate failure extropy
under the transformations $X_{i}\rightarrow \psi(X_{i})$,~i=1,2.
\begin{proposition}\label{prop2.6}
	Let $(Y_{1},Y_{2})$ be a nonnegative and absolutely continuous random vector. Further, let $Y_{i}=\psi(X_{i}),~i=1,2$ be one-to-one transformations such that $\psi(x_{i})$'s are differentiable. Then,
	\begin{eqnarray}
	\mathcal{E}_{f}(Y_1,Y_2)=\frac{1}{4}\int_{0}^{\infty}\int_{0}^{\infty}K^{2}(x_1,x_2)|J|dx_1dx_2,
	\end{eqnarray}
	where $J=\frac{\partial}{\partial x_1}\psi_{1}(x_{1})\frac{\partial}{\partial x_2}\psi_{2}(x_{2})$ is the Jacobian of the transformations.
\end{proposition}
\begin{proof}
	The proof is straightforward. Thus, it is omitted.
\end{proof}

Consider the conditional random variable $X|Y$ with cdf $K_{X|Y}(x|y)=\frac{K(x,y)}{G(y)}.$ Then, the conditional failure extropy is given by
\begin{eqnarray}\label{eq2.7}
\mathcal{E}_{f}(X|Y)=-\frac{1}{2}\int_{0}^{\infty}K^{2}_{X|Y}(x|y)dx.
\end{eqnarray}
Note that when $X$ and $Y$ are independent, then $K_{X|Y}(x|y)=F(x).$ Thus, clearly $$\mathcal{E}_{f}(X|Y)=\mathcal{E}_{f}(X).$$

\section{Dynamic failure extropy\setcounter{equation}{0}}
In order to illustrate the effect of the age t on the uncertainty in random variables, dynamic
information measures play a useful role in reliability.  This section introduces dynamic version of the
failure extropy. Suppose that a system is working. It is pre-planned that the system will be inspected at time $t>0$. At the inspection time,  the system is down. Thus, uncertainty relies in the inactivity (past) time $(t-X|X<t).$ For the random variable $X$, the dynamic (past)
failure extropy is defined as (see \cite{kundu2020cumulative} and \cite{nair2020dynamic})
\begin{eqnarray}\label{eq3.1}
\mathcal{E}_{f}(X;t)=-\frac{1}{2}\int_{0}^{t}\left(\frac{F(x)}{F(t)}\right)^{2}dx,~t>0,~F(t)>0.
\end{eqnarray}
Under the condition that at time $t$, the system is down, $\mathcal{E}_{f}(X;t)$ quantifies the uncertainty about its past life. Clearly, $\mathcal{E}_{f}(X;0)=0$ and
$\mathcal{E}_{f}(X;\infty)=\mathcal{E}_{f}(X)$.
One can easily check that like the failure extropy, the dynamic failure
extropy (DFE) takes negative values. The DFE
of $X$ is equal to the failure extropy of the inactivity time.  Now, we consider some distributions and obtain closed form expressions of the DFE.
\begin{example}~\\
	\vspace{-.5cm}
\begin{itemize}
\item Let $X$ follow uniform distribution in
the interval $(0,b)$. Thus, $F(x)=x/b$, $0<x<b$.
Then, $\mathcal{E}_{f}(X;t)=-t/6$, for $0<t<b$.
\item Let $X$ have type-III extreme value
distribution with cdf
$F(x)=\exp\{\alpha(x-\beta)\}$, $x<\beta$ and
$\alpha>0$. Then,
$\mathcal{E}_{f}(X;t)=-(\exp\{2\alpha
t\}-1)/(4\alpha\exp\{2\alpha t\}),$ for $t<\beta.$
\item Consider a Pareto random variable $X$ with
cdf $F(x)=1-(b/x)^{a},~x>b>0,~a>0$. We obtain
$\mathcal{E}_{f}(X;t)=-\frac{1}{2(1-(b/t)^{a})^{2}}[(t-b)+\frac{2b^{a}}{a-1}(\frac{1}{t^{a-1}}
-\frac{1}{b^{a-1}})-\frac{b^{2a}}{2a-1}(\frac{1}{t^{2a-1}}-\frac{1}{b^{2a-1}})],$ for $a\ne 1,\frac{1}{2}$ and $t>b.$
\end{itemize}
\end{example}

We also consider exponential distribution with cdf $F(x)=1-e^{-\lambda x},~x>0,~\lambda>0$ to compute the DFE. This is given by
\begin{eqnarray}
\mathcal{E}_{f}(X;t)=-\frac{1}{2(1-e^{-\lambda t})^{2}}\left[t+\frac{2e^{-\lambda t}}{\lambda}-\frac{e^{-2\lambda t}}{2\lambda}-\frac{2}{\lambda}+\frac{1}{2\lambda}\right].
\end{eqnarray}
For the purpose of visualization, we have depicted the DFE of exponential distribution for several values of the parameter $\lambda$ in Figure $1$.
\begin{figure}[h]
\begin{center}
\includegraphics[height=3.8in]{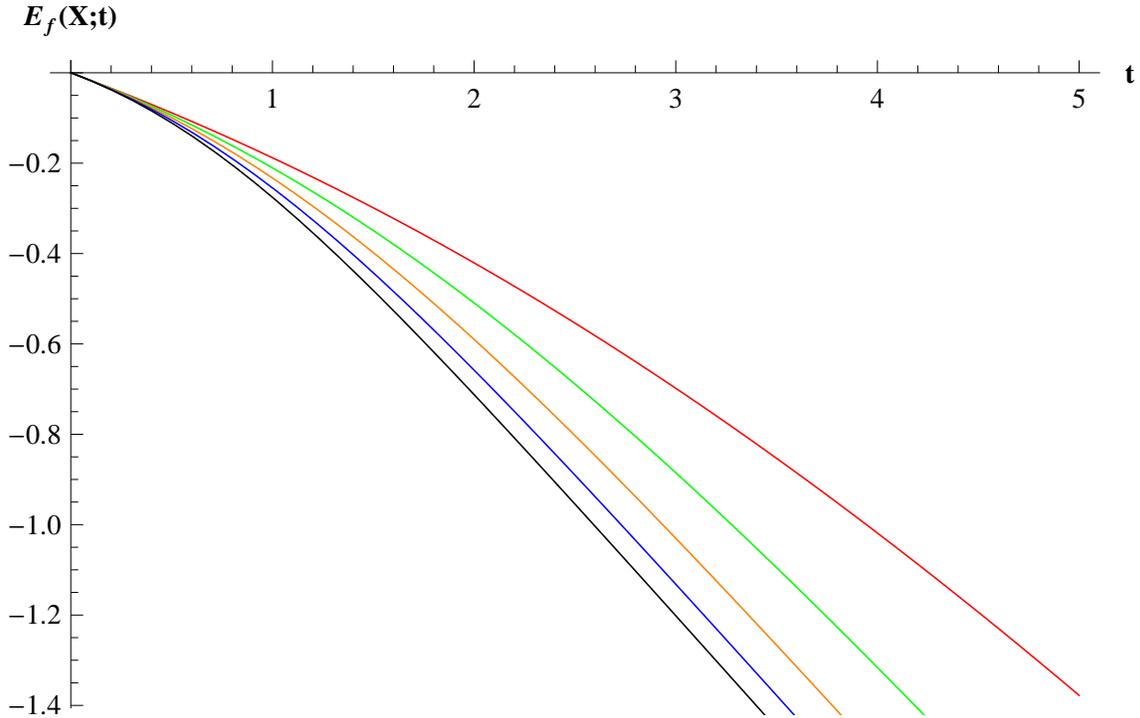}
\end{center}
\caption{Plots of the dynamic failure extropy of the exponential distribution for various values of $\lambda.$ The values of $\lambda$ are taken as $0.5,~1,~1.5,~2$ and $2.5$ from above..}
\end{figure}
Similarly to Proposition \ref{prop2.1}, here, we obtain lower and upper bounds of the dynamic failure extropy. The proof is omitted.
\begin{proposition}
Let $X$ be a nonnegative and absolutely continuous random variable with cdf $F.$ Further, denote the mean inactivity time of $X$ as $m_{I}(X;t)=E[t-X|X<t]=\int_{0}^{t}\frac{F(x)}{F(t)}dx.$ Then,
\begin{itemize}
	\item $\mathcal{E}_{f}(X;t)\ge -\frac{1}{2F^{2}(t)}\int_{0}^{t}F(x)dx;$
	\item $\mathcal{E}_{f}(X;t)\le -\frac{1}{2t}m_{I}^{2}(X;t).$
\end{itemize}
\end{proposition}

In the following theorem, we study the effect of monotone transformations on the DFE. The proof is similar to Theorem \ref{th2.1} and thus, it is omitted.
\begin{theorem}\label{th3.1}
	Let $X$ be a nonnegative absolutely continuous
	random variable with cdf $F$. Further, let $Y=\psi(X)$ be strictly
	monotone and differentiable. Denote the dynamic survival extropy
	of $Y$ by $\mathcal{E}_{f}(Y;t).$ Then,
	\begin{eqnarray*}
		\mathcal{E}_{f}(Y;t)=\begin{cases}
			-\frac{1}{2}\int_{\psi^{-1}(0)}^{\psi^{-1}(t)}\psi'(x)\frac{F^{2}(x)}{F^{2}(\psi^{-1}(t))}dx, &
			if~\psi~
			\mbox{is strictly increasing;} \\
			-\frac{1}{2}\int_{\psi^{-1}(t)}^{\psi^{-1}(0)}|\psi'(x)|\frac{F^{2}(x)}{F^{2}(\psi^{-1}(t))}dx,
			& if~\psi~ \mbox{is strictly decreasing.}
		\end{cases}
	\end{eqnarray*}
\end{theorem}

Below, we present the effect of the
transformation $Y=\alpha X+\beta$, where
$\alpha>0$ and $\beta\ge0$ on the DFE. It immediately follows from Theorem \ref{th3.1}.
\begin{proposition}\label{prop3.2}
If $Y=\alpha X+\beta$, where $\alpha>0$ and
$\beta\ge0$, then $\mathcal{E}_{f}(Y;t)=\alpha
\mathcal{E}_{f}(X;\frac{t-\beta}{\alpha})$.
\end{proposition}

\cite{abdul2019dynamic} proposed  ordering based on the dynamic survival extropy.  Here, we consider ordering in terms of the DFE given by (\ref{eq3.1}). One may refer to \cite{kundu2020cumulative} and \cite{nair2020dynamic} for similar order.
\begin{definition}
	Let $X$ and $Y$ be two nonnegative and absolutely continuous random variables with cdfs $F$ and $G$, respectively. Then, $X$ is said to be samller than $Y$ in DFE, abbreviated by $X\le_{dfe}Y$, if $\mathcal{E}_{f}(X;t)\le \mathcal{E}_{f}(Y;t)$, for all $t>0.$
\end{definition}
This definition of DFE ordering reveals that the uncertainty in $(X|X<t)$ about the predictability of the future time is smaller than that of $(Y|Y<t)$. Similar to the failure extropy order, the DFE order is also a partial order.
Below, we obtain a connection between the reversed hazard rate ordering and the DFE ordering (also see \cite{nair2020dynamic}).
\begin{theorem}\label{th3.2}
	Let $X$ and $Y$ be two nonnegative and absolutely continuous random variables with cdfs $F$ and $G$, respectively. Then, $$X\le_{rh}Y\Rightarrow X\le_{dfe}Y.$$
\end{theorem}

Now, we consider an example, which is an application of  Theorem \ref{th3.2}.
\begin{example}
	Let $X_{1},\ldots,X_{n}$ be independent identical observations with reversed hazard rate $r$. Denote $X_{n:n}=\max\{X_{1},\ldots,X_{n}\}$ and its reversed hazard rate by $r_{n:n}(t)=nr(t).$ Then,  $X_{1}\le_{rh}X_{n:n}$. Thus, an application of Theorem \ref{th3.2} produces $\mathcal{E}_{f}(X_1;t)\le \mathcal{E}_{f}(X_{n:n};t).$
\end{example}
\begin{remark}
Let $\xi$ be continuous and strictly increasing function. Then, from \cite{nanda2001hazard}, we have
\begin{eqnarray}
X\le_{rh}Y\Rightarrow \xi(X)\le_{rh}\xi(Y)\Rightarrow \mathcal{E}_{f}(\xi(X);t)\le \mathcal{E}_{f}(\xi(Y);t)\Rightarrow   \xi(X)\le_{dfe} \xi(Y).
\end{eqnarray}
\end{remark}

Basically, transformations are used to transform one probability density function into another.
So, it is of interest to study if the uncertainty order (here $X\le_{dfe}Y$) is preserved under some
transformations. Next proposition provides sufficient conditions for preserving the ordering $X\le_{dfe}Y$
under increasing (affine) transformations. Next result characterizes DFE ordering through increasing transformation.
\begin{proposition}
	Let $X$ and $Y$ be two nonnegative and absolutely continuous random variables such that $X\le_{dfe}Y$. Let us take $Z_{1}=\alpha_{1}X+\beta_{1}$ and $Z_{2}=\alpha_{2}Y+\beta_{2}$ such that $0<\alpha_1\le \alpha_2$ and $0<\beta_1\le \beta_2.$ Then, $Z_{1}\le_{dfe}Z_{2}$, provided $\mathcal{E}_{f}(X;t)$ is decreasing with respect to $t~(>\beta_2)$.
\end{proposition}
\begin{proof}
	The proof follows under the assumptions made and Proposition \ref{prop3.2}. Thus, it is omitted.
\end{proof}

We propose a nonparametric class of distributions.
\begin{definition}
A random variable $X$ with cdf $F$ belongs to the
class of decreasing dynamic failure
extropy (DDFE) if
$\mathcal{E}_{f}(X;t)$ is decreasing
with respect to $t.$
\end{definition}
Clearly, uniform distribution belongs to this class. Next, we find conditions, under which the reversed hazard rate is decreasing.
Note that Equation (\ref{eq3.1}) can be
alternatively expressed as
\begin{eqnarray}\label{eq3.2}
[F(t)]^{2}\mathcal{E}_{f}(X;t)=-\frac{1}{2}\int_{0}^{t}F^{2}(x)dx.
\end{eqnarray}
On differentiating (\ref{eq3.2}) with respect to
$t$, we obtain
\begin{eqnarray}\label{eq3.3}
\frac{d}{dt}\mathcal{E}_{f}(X;t)+2r(t)\mathcal{E}_{f}(X;t)+\frac{1}{2}=0,
\end{eqnarray}
where $r(t)=f(t)/F(t)$ is known as the reversed
hazard rate of $X$. Further, differentiating (\ref{eq3.3}) with respect to $t$, we get
\begin{eqnarray}\label{eq3.5}
\frac{d^2}{dt^2}\mathcal{E}_{f}(X;t)+2r'(t)\mathcal{E}_{f}(X;t)+2r(t)\frac{d}{dt}\mathcal{E}_{f}(X;t)=0.
\end{eqnarray}
Thus, we have the following result.
\begin{proposition}\label{prop3.4}
Let $X$ be a nonnegative and absolutely continuous random variable. If $\mathcal{E}_{f}(X;t)$ is decreasing and concave, then $X$ has decreasing reversed hazard rate function.
\end{proposition}
The following proposition provides an upper bound of the reversed hazard rate function in terms of the DFE, when $X$ belongs to the DDFE class.
\begin{proposition}
Let $X$ be an absolutely continuous and
nonnegative random variable. Further, let
$\mathcal{E}_{f}(X;t)$ be decreasing with respect
to $t$. Then, $r(t) \le -\frac{1}{4\mathcal{E}_{f}(X;t)}$.
\end{proposition}
\begin{proof}
The proof follows from (\ref{eq3.3}).
\end{proof}

It is always of interest to deal with the problem of characterizing probability distributions. The
general characterization problem is to obtain conditions under which the DFE uniquely determines
the distribution function. Next theorem shows that the DFE determines the distribution function uniquely. See \cite{nair2020dynamic} for the proof.
\begin{theorem}
	Let $X$ have the cdf $F$ and reversed hazard rate function $r.$ Then, $\mathcal{E}_{f}(X;t)$ uniquely determines the cdf.
\end{theorem}


Let $(X_{1},X_{2})$ be the lifetimes of two units of a system. Assume that the first and second components are found failed at respective times $t_1$ and $t_2$. Then, the unceratinty associated with the past lifetimes of the system can be quantified by  the two dimensional extension of the DFE. It is given by
\begin{eqnarray}
\mathcal{E}_{f}(X_{1},X_{2};t_{1},t_{2})=\frac{1}{4}\int_{0}^{t_{1}}\int_{0}^{t_{2}}\left(\frac{K(x_{1},x_{2})}{K(t_1,t_2)}\right)^{2}dx_{1}dx_{2}.
\end{eqnarray}
Further, let $X_{1}$ and $X_{2}$ be independent. Therefore,
\begin{eqnarray}
\mathcal{E}_{f}(X_{1},X_{2};t_{1},t_{2})=\mathcal{E}_{f}(X_{1};t_{1})\mathcal{E}_{f}(X_{2};t_{2}).
\end{eqnarray}
Similar to Proposition \ref{prop2.6}, it can be shown that the bivariate DFE is also not invariant under the nonsingular transformations $Y_{i}=\psi(X_{i}),~i=1,2$.

\section{Weighted (dynamic) failure extropy\setcounter{equation}{0}}
We now turn to the weighted versions of the failure extropy and DFE. In \cite{balakrishnan2020weighted}, the authors studied the weighted versions of extropy and past extropy. First, let us discuss the weighted failure extropy. Note that the weighted distributions were introduced in the literature to get better fitted model of a dataset. Let
$X$ be a nonnegative absolutely continuous random
variable with cdf $F$ and pdf $f$. Then, the
weighted (length-biased) random variable associated with the
random variable $X$, denoted by $X_{w}$ has the pdf
$f_{w}(t)=(w(t)f(t))/\mu_{w}$ and cdf
$F_{w}(t)=(E[w(X)|X\le t]F(t))/\mu_{w}$, where
$w(x)~(\ge0)$ is called weight function and
$\mu_{w}=E(w(X))<\infty.$ Utilizing these, the
failure extropy of $X_{w}$ is defined as
\begin{eqnarray}
\mathcal{E}_{f}^{w}(X)&=&-\frac{1}{2}\int_{0}^{\infty}F^{2}_{w}(x)dx\nonumber\\
&=&
-\frac{1}{2\mu_{w}^{2}}\int_{0}^{\infty}\{E[w(X)|X\le
x]F(x)\}^{2}dx.
\end{eqnarray}
This is known as the weighted failure extropy (WFE). It
takes negative values. Note that the WFE becomes the length-biased failure
extropy when $w(x)=x.$ Let us denote the
length-biased random variable by $X_{l}.$ Then,
the pdf of $X_{l}$ is $f_{l}(t)=(tf(t))/\mu_{l}$
and cdf is $F_{l}(t)=(E[X|X\le t]F(t))/\mu_{l},$
where $\mu_{l}=E(X).$ The failure extropy of
$X_{l}$ is given by
\begin{eqnarray}
\mathcal{E}_{f}^{l}(X)&=&-\frac{1}{2}\int_{0}^{\infty}F^{2}_{l}(x)dx\nonumber\\
&=&
-\frac{1}{2\mu_{l}^{2}}\int_{0}^{\infty}\{E[X|X\le
x]F(x)\}^{2}dx.
\end{eqnarray}
The following example illustrates the importance
of the weighted (length-biased) failure extropy.
\begin{example}
Consider two random variables $X$ and $Y$ with
respective cdfs
\begin{eqnarray}
F(x)=\begin{cases} 0,& if~x\le2\\\frac{x-2}{2}, & if~ 2<x<4 \\
1, & if~x\ge 4
\end{cases}
\end{eqnarray}
and
\begin{eqnarray}
G(x)=\begin{cases} 0,& if~x\le4\\\frac{x-4}{2}, & if~ 4<x<6 \\
1, & if~x\ge 6.
\end{cases}
\end{eqnarray}
It can be computed that
$\mathcal{E}_{f}(X)=\mathcal{E}_{f}(Y)=-\frac{1}{3},$
that is, the expected uncertainties contained in
both distributions are the same. However, their
associated uncertainties quantified via weighted
(length-biased) failure entropy are not the same.
These are obtained as
$\mathcal{E}_{f}^{l}(X)\approx -0.2816$ and
$\mathcal{E}_{f}^{l}(Y)\approx -0.813$.
\end{example}
The following proposition provides a lower bound
of the weighted failure extropy based on the
failure extropy. The similar statement also holds for the case of the length-biased failure extropy. The proof is straightforward and
thus, it is omitted.
\begin{proposition}
Let us assume that $E[w(X)|X\le x]\le E[w(X)].$
Then, $\mathcal{E}_{f}^{w}(X)\ge
\mathcal{E}_{f}(X).$
\end{proposition}

In Section $2$, the inequalities of the failure extropies of the random variables $X_{\tau}^*$ and $X$ are obtained. Here, we try to obtain similar inequalities for the weighted version of the failure extropies of  $X_{\tau}^*$ and $X$. We denote the weighted random variable associated with $X_{\tau}^*$ by $X_{w}^*.$ The cdf of $X_{w}^*$ is given by $F_{w}^{*}(x)=(E[w(X)|X\le x]F^{\tau}(x))/\mu_{w}$.
Thus, we have the following result. Its proof is similar to Proposition \ref{prop2.2}, and thus it is omitted.
\begin{proposition}
	Let $X$ be a nonnegative and absolutely continuous random variable with cdf $F$. Further, let $X_{\tau}^*$ be another random variable with cdf  $F^*(x)=[F(x)]^{\tau},~\tau>0.$ Then, $\mathcal{E}_{f}^{w}(X_{\tau}^*)\ge (\le) \mathcal{E}_{f}^{w}(X)$ for $\tau>(<)1$.
\end{proposition}

Next, we introduce weighted form of the DFE. For the weighted random variable $X_{w}$,
the DFE is defined as
\begin{eqnarray}\label{eq4.3}
\mathcal{E}_{f}^{w}(X;t)&=&-\frac{1}{2}\int_{0}^{t}\frac{F^{2}_{w}(x)}{F^{2}_{w}(t)}dx \nonumber\\
&=& -\frac{1}{2}\int_{0}^{t}\frac{(E[w(X)|X\le
x]F(x))^{2}}{(E[w(X)|X\le t]F(t))^{2}}dx.
\end{eqnarray}
This measure is known as the weighted DFE of the random variable $X$. Further, the length-biased DFE can be
obtained from (\ref{eq4.3}) by substituting
$w(x)=x.$ It is given by
\begin{eqnarray}\label{eq4.4}
\mathcal{E}_{f}^{l}(X;t)&=&-\frac{1}{2}\int_{0}^{t}\frac{F^{2}_{l}(x)}{F^{2}_{l}(t)}dx \nonumber\\
&=& -\frac{1}{2}\int_{0}^{t}\frac{(E[X|X\le
x]F(x))^{2}}{(E[X|X\le t]F(t))^{2}}dx.
\end{eqnarray}
Differentiating (\ref{eq4.3}) with respect to
$t$, we obtain
\begin{eqnarray}\label{eq4.5}
\frac{d}{dt}\mathcal{E}_{f}^{w}(X;t)=-2
r_{f}^{w}(t)\mathcal{E}_{f}^{w}(X;t)-\frac{1}{2},
\end{eqnarray}
where $r_{f}^{w}(t)=f_{w}(t)/F_{w}(t)$. Thus, we
have the following proposition dealing with the
bound of $\mathcal{E}_{f}^{w}(X;t).$
\begin{proposition}
Let $\mathcal{E}_{f}^{w}(X;t)$ be
decreasing with respect to $t$. Then,
$\mathcal{E}_{f}^{w}(X;t)\ge-1/(4r_{f}^{w}(t)).$
\end{proposition}
\begin{proof}
The proof of the proposition is immediate from
(\ref{eq4.5}). Thus, it is omitted.
\end{proof}
Similar to Proposition \ref{prop3.4}, we get the following result for the reversed hazard rate of the weighted random variable.
\begin{proposition}
	Let $\mathcal{E}_{f}^{w}(X;t)$ be decreasing and concave. Then, the weighted random variable $X_{w}$ has decreasing reversed hazard rate.
\end{proposition}

\begin{proposition}
If $E[w(X)|X\le x]\le (\ge) E[w(X)|X\le t]$ for $x\le t$, then $\mathcal{E}_{f}^{w}(X;t)\ge (\le) \mathcal{E}_{f}(X;t)$.
\end{proposition}
\begin{proof}
	The proof is straightforward and thus, it is omitted.
\end{proof}
The following corollary is immediate from the above proposition.
\begin{corollary}
	If $E[X|X\le x]\le (\ge) E[X|X\le t]$ for $x\le t$, then $\mathcal{E}_{f}^{l}(X;t)\ge (\le) \mathcal{E}_{f}(X;t)$.
\end{corollary}

Now, we provide the effect of the scale
transformation to the length-biased DFE.
\begin{proposition}
Let $Y=aX$, where $a>0$. Then,
$\mathcal{E}_{f}^{l}(Y;t)=a\mathcal{E}_{f}^{l}(X;t/a).$
\end{proposition}

Now, we define a nonparametric class based on the weighted DFE.
\begin{definition}
	A random variable $X$ with cdf $F$ is said to have decreasing weighted DFE, denoted by DWDFE if $\mathcal{E}_{f}^{w}(X;t)$ is decreasing with respect to $t>0.$
\end{definition}


\section{Estimation\setcounter{equation}{0}}
In this section, we propose nonparametric estimators for the failure extropy, DFE and its weighted versions based on the empirical cumulative distribution function. Consider a simple random sample $(X_{1},\ldots,X_{n})$ of size $n$ from a population with cdf $F.$ The empirical distribution function of the sample is
\begin{eqnarray}
\tilde{F}_{n}(x)=\frac{1}{n}\sum_{i=1}^{n}I(X_{i}\le x),~~x\in \mathbb{R},
\end{eqnarray}
where $I$ is the indicator function. Denote the order statistics of $(X_{1},\ldots,X_{n})$ by $X_{(1)}\le\ldots\le X_{(i)}\le\ldots\le X_{(n)},$ where $X_{(i)}$ is the $i$th order statistic, $i=1,\ldots,n.$ Again,
\begin{eqnarray*}
\tilde{F}_{n}(x) = \displaystyle\left\{\begin{array}{ll} 0,
& \textrm{ $x<x_{(1)},$}\\
\frac{j}{n},& \textrm{$x_{(j)}\leq x<x_{(j+1)},~j=1,2,\ldots,
n-1,$}\\
1,& \textrm{$x\geq x_{(n)},$}
\end{array} \right.
\end{eqnarray*}
where $x_{(i)}$ is the realization of $X_{(i)}$. Thus, the empirical estimate of the failure extropy is
\begin{eqnarray}
\mathcal{E}_{f}(\tilde{F}_{n})&=&-\frac{1}{2}\int_{0}^{\sup S_{X}}\tilde{F}_{n}^{2}(x)dx\nonumber\\
&=&-\frac{1}{2}\sum_{j=1}^{n-1}\int_{x_{(j)}}^{x_{(j+1)}}\tilde{F}_{n}^{2}(x)dx\nonumber\\
&=&-\frac{1}{2}\sum_{j=1}^{n-1}\left(\frac{j}{n}\right)^{2}(x_{(j+1)}-x_{(j)}).
\end{eqnarray}
We consider two examples (see also \cite{kayal2016generalized} and \cite{di2017further}) which are dealing with exponential and uniform distributions.
\begin{example}
	Let $X_{1},\ldots,X_{n}$ be a random sample from exponential distribution with mean $1/\lambda,$ where $\lambda>0.$ Then, because of the independent sample spacings, $X_{(j+1)}-X_{(j)}$ follows exponential distribution with mean $1/(\lambda(n-j)).$ Thus, we have
	\begin{eqnarray}
	E[\mathcal{E}_{f}(\tilde{F}_{n})]=-\frac{1}{2}\sum_{j=1}^{n-1}\left(\frac{j}{n}\right)^{2}\frac{1}{\lambda(n-j)}
	\end{eqnarray}
	and
	\begin{eqnarray}
	 Var[\mathcal{E}_{f}(\tilde{F}_{n})]=\frac{1}{4}\sum_{j=1}^{n-1}\left(\frac{j}{n}\right)^{4}\frac{1}{\lambda^{2}(n-j)^{2}}.
	\end{eqnarray}
\end{example}
\begin{example}\label{ex5.2}
	Consider a random sample $X_{1},\ldots,X_{n}$ from the uniform distribution in $(0,1)$. Then,  $X_{(j+1)}-X_{(j)}$ follows beta distribution with parameters $1$ and $n$ with mean $E[X_{(j+1)}-X_{(j)}]=1/(n+1).$ Then.
	\begin{eqnarray}
	E[\mathcal{E}_{f}(\tilde{F}_{n})]=-\frac{1}{2}\sum_{j=1}^{n-1}\left(\frac{j}{n}\right)^{2}\frac{1}{(n+1)}
	\end{eqnarray}
	and
	\begin{eqnarray}
	 Var[\mathcal{E}_{f}(\tilde{F}_{n})]=\frac{1}{4}\sum_{j=1}^{n-1}\left(\frac{j}{n}\right)^{4}\frac{n}{(n+1)^{2}(n+2)}.
	\end{eqnarray}
\end{example}
\begin{remark}
Clearly, from Example \ref{ex5.2}, we have $$\lim_{n\rightarrow \infty}E[\mathcal{E}_{f}(\tilde{F}_{n})]=-1/6~~ and ~~\lim_{n\rightarrow \infty}Var[\mathcal{E}_{f}(\tilde{F}_{n})]=0.$$ This implies that the nonparametric estimator of the failure extropy is consistent when the random sample is taken from $U(0,1)$
distribution.
\end{remark}

Now, for the case of the DFE, we assume that $x_{(1)}\le\ldots\le x_{(k)}\le t \le x_{(k+1)}\le \ldots \le x_{(n)},$ where $k=1,\ldots,n-1.$ Thus, the empirical DFE is given by
\begin{eqnarray}
\mathcal{E}_{f}(\tilde{F}_{n};t)&=&-\frac{1}{2}\left[\sum_{j=1}^{k-1}\left(\frac{j}{k}\right)^{2}(x_{(j+1)}-x_{(j)})
+(t-x_{(k)})\right].
\end{eqnarray}
Further, the nonparametric estimate of $\mathcal{E}_{f}^{w}(X)$ and $\mathcal{E}_{f}^{w}(X;t)$ are respectively given by
\begin{eqnarray}\label{eq5.8}
\mathcal{E}_{f}^{w}(\tilde{F}_{n})&=&-\frac{1}{2}\left[\sum_{j=1}^{n-1}\left(\frac{j\sum_{i=1}^{j}w(x_{(i)})}{n\sum_{i=1}^{n}w(x_{(i)})}\right)^{2}(x_{(j+1)}-x_{(j)})
	\right],\\
	 \mathcal{E}_{f}^{w}(\tilde{F}_{n};t)&=&-\frac{1}{2}\left[\sum_{j=1}^{k-1}\left(\frac{j\sum_{i=1}^{j}w(x_{(i)})}{k\sum_{i=1}^{k}w(x_{(i)})}\right)^{2}(x_{(j+1)}-x_{(j)})+(t-x_{(k)})\right].
	\end{eqnarray}
Next, we consider two real life data sets. The first consists of  lifetimes (in days) of blood cancer patients (see \cite{abu2002moment}) and the second contains  spike times (in microseconds) of a neuron (see \cite{kass2003statistical}).
\begin{example}\label{ex5.3}~\\\\
{\bf (i)~ Lifetimes of blood cancer patients:-}\\ $115,181,255,418,441,461,516,739,743,789,807,865,924,983,1024,1062,1063,1165,1191,\\
1222,1222,1251,1277,1290,1357,1369,1408,1455,1478,1549,1578,1578,1599,1603,1605,\\1696,
1735,1799,1815,1852.$ \\ \\
{\bf (ii)~ Spike times observed in $8$
trials on a single neuron:-}\\
$
136.842,145.965,155.088,175.439,184.561,199.298,221.053,231.579,246.316,263.158,
274.386,\\282.105,317.193,329.123,347.368,360.702,368.421,389.474,392.982,432.281,449.123,
463.86,\\503.86,538.947,586.667,596.491,658.246,668.772,684.912. $\\
\\
From the first set of data, we have $\mathcal{E}_{f}(\tilde{F}_{n})=-222.752$,  $\mathcal{E}_{f}^{w}(\tilde{F}_{n})=-104.13$, $\mathcal{E}_{f}(\tilde{F}_{n};t=1000)=-128.069$ and $\mathcal{E}_{f}^{w}(\tilde{F}_{n};t=1000)=-55.237.$ Further, from the second dataset, we compute $\mathcal{E}_{f}(\tilde{F}_{n})=-113.135$,  $\mathcal{E}_{f}^{w}(\tilde{F}_{n})=-50.9208$, $\mathcal{E}_{f}(\tilde{F}_{n};t=340)=-39.8138$ and $\mathcal{E}_{f}^{w}(\tilde{F}_{n};t=340)=-21.999.$
\end{example}
\begin{figure}[h]
\begin{center}
\subfigure[]{\label{c1}\includegraphics[height=2in]{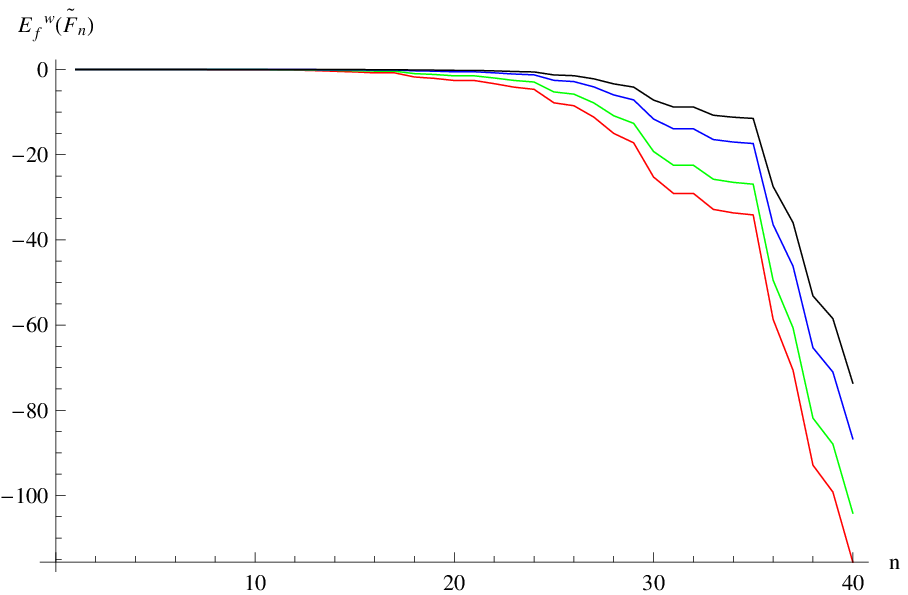}}
\subfigure[]{\label{c1}\includegraphics[height=2in]{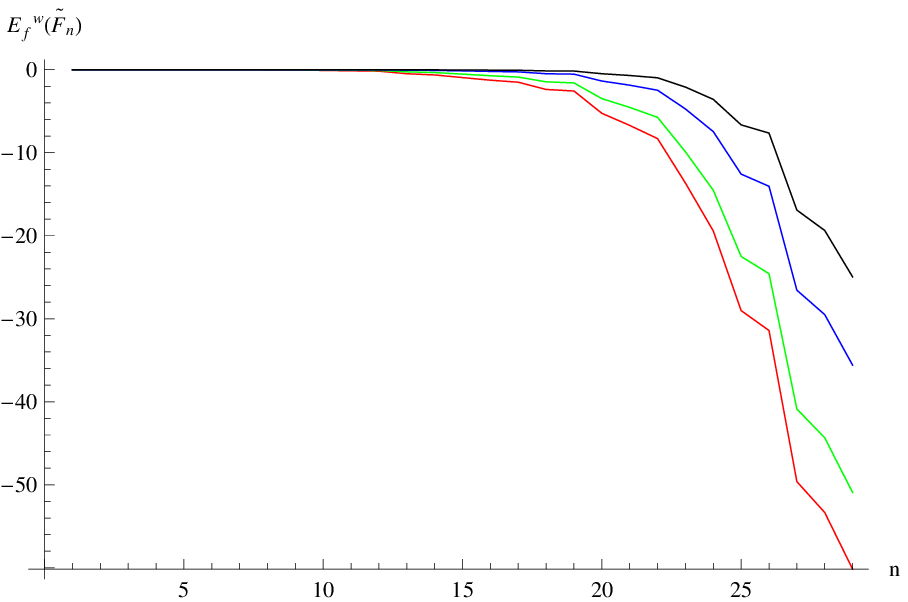}}
\caption{(a) Plot of the empirical estimator given by (\ref{eq5.8}) for the the first dataset in Example \ref{ex5.3} for various weight functions.  The weight functions are taken as $w(x)=\sqrt{x},~x,~x^2$ and $x^{3}$ (from below to above). (b) Plot of the empirical estimator given by (\ref{eq5.8}) for the the second dataset in Example \ref{ex5.3} for various weight functions.  The weight functions are taken as $w(x)=\sqrt{x},~x,~x^2$ and $x^{3}$ (from below to above).} \label{fig2}
\end{center}
\end{figure}
\section{Concluding remarks}
Motivated by the concepts of the cumulative entropy and the failure entropy, in this paper, we have proposed failure extropy, dynamic failure extropy and their weighted versions. For failure extropy, we investigated various properties. Specifically, bounds, effect of monotone transformations and two dimensional version are studied. The uncertainty order based on the failure extropy has been introduced. Connection to other stochastic orders are also investigated. It has been shown that the usual stochastic order implies the failure extropy order. Since the failure extropy is not an useful tool to quantify uncertainty, which relies in a failed system, we also introduced a dynamic version of it. Similar properties have been discussed for this dynamic measure. A nonparametric class and some characterization results have been developed. Further, we have defined these measures for the weighted random variables. Various properties have been studied. Finally, empirical estimators of the proposed measures have been provided and computed numerically from two real life datasets.\\
\\
\\
{\bf Acknowledgements:} The author would like to thank the Associate Editor and an anonymous reviewer for their positive remarks and useful comments. Suchandan Kayal gratefully acknowledges the partial financial support for this research work under a Grant MTR/2018/000350, SERB, India.
\bibliography{K-GLFR_bib}

\end{document}